\documentclass{amsart}
\usepackage{amssymb}
\usepackage{amsfonts}
\usepackage{hyperref}
\usepackage{cite}

\newtheorem{theorem}{Theorem}
\theoremstyle{plain}

\newtheorem{corollary}{Corollary}

\newtheorem{notation}{Notation}

\numberwithin{equation}{section}

\begin{document}
\title[Einstein-like doubly warped product manifolds]{Gray's decomposition
on doubly warped product manifolds and applications}
\author{Hoda K. El-Sayied}
\address{Mathematics Department, Faculty of Science, Tanata University,
Tanta, Egypt}
\email{hkelsayied1989@yahoo.com}
\author{Carlo A. Mantica}
\address[C. Mantica]{I.I.S. Lagrange, Via L. Modignani 65, 20161 Milan,
Italy,}
\email{carloalberto.mantica@libero.it}
\author{Sameh Shenawy}
\email[S. Shenawy]{drssshenawy@eng.modern-academy.edu.eg}
\address{Modern Academy for engineering and Technology, Maadi, Egypt}
\urladdr{http://www.modern-academy.edu.eg}
\author{N. Syied}
\curraddr{Modern Academy for engineering and Technology, Maadi, Egypt}
\email[Noha Syied]{drnsyied@mail.com, drnsayed@gmail.com}
\urladdr{http://www.modern-academy.edu.eg}
\subjclass[2010]{Primary 53C21; Secondary 53C50, 53C80}
\keywords{Codazzi Ricci tensor, doubly warped manifolds, Killing Ricci
tensor, doubly warped space-times, Einstein-like manifolds.}

\begin{abstract}
A. Gray presented an interesting $O\left( n\right) $ invariant decomposition
of the covariant derivative of the Ricci tensor. Manifolds whose Ricci
tensor satisfies the defining property of each orthogonal class are called
Einstein-like manifolds. In the present paper, we answered the following
question: Under what condition(s), does a factor manifold $M_{i},i=1,2$ of a
doubly warped product manifold $M=_{f_{2}}M_{1}\times _{f_{1}}M_{2}$ lie in
the same Einstein-like class of $M$? By imposing sufficient and necessary
conditions on the warping functions, an inheritance property of each class
is proved. As an application, Einstein-like doubly warped product
space-times of type $\mathcal{A},$ $\mathcal{B}$ or $\mathcal{P}$ are
considered.
\end{abstract}

\maketitle

\section{An introduction}

Alfred Gray in \cite{Gray:1978} presented $O\left( n\right) $ invariant
orthogonal irreducible decomposition of the space $W$ of all $\left(
0,3\right) $ tensors satisfying only the identities of the gradient of the
Ricci tensor $\nabla _{k}R_{ij}$. The space $W$ is decomposed into three
orthogonal irreducible subspaces, that is, $W=\mathcal{A\oplus B\oplus I}$.
This decomposition produces seven classes of Einstein-like manifolds, that
is, manifolds whose Ricci tensor satisfies the defining identity of each
subspace. They are the trivial class $\mathcal{P}$, the classes $\mathcal{A}$%
, $\mathcal{B}$, $\mathcal{I}$ and three composite classes $\mathcal{I\oplus
A}$, $\mathcal{I\oplus B}$ and $\mathcal{A\oplus B}$.

In class $\mathcal{P}$, the Ricci tensor is parallel i.e. $\nabla
_{k}R_{ij}=0$ whereas class $\mathcal{A}$ contains manifolds whose Ricci
tensor is Killing. The Ricci tensor of manifolds in class $\mathcal{B}$ is a
Codazzi tensor i.e. $\nabla _{k}R_{ij}=\nabla _{i}R_{kj}$. The traceless
part of the Ricci tensor vanishes in class $\mathcal{I}$ i.e. class $%
\mathcal{I}$\ contains Sinyukov manifolds\cite{Mantica:2019}. The tensor%
\begin{equation*}
\mathcal{L}_{ij}=R_{ij}-\frac{2R}{n+2}g_{ij}
\end{equation*}%
is Killing in class $\mathcal{I\oplus A}$ whereas the tensor%
\begin{equation*}
\mathcal{H}_{ij}=R_{ij}-\frac{R}{2\left( n-1\right) }g_{ij}
\end{equation*}%
is a Codazzi tensor in class $\mathcal{I\oplus B}$. The class $\mathcal{%
A\oplus B}$ is identified by having constant scalar curvature. The same
decomposition is discussed extensively in \cite[Chapter 16]{Besse:2008} (see
also \cite{Mantica:2012, Mantica:2019} and Section 3 for more details and
equivalent conditions). Thereafter, Einstein-like manifolds have been
studied by many authors such as G. Calvaruso in \cite%
{Calvaruso:2007,Calvaruso:2008,Calvaruso:2009,Calvaruso:2011} Mantica et al
in \cite{Mantica:2012, Mantica:2017, Mantica:2019} and many others \cite%
{Berndt:1992,Boeckx:1992,Bueken:1999,Peng:2016,Zaeim:2016}. An interesting
study in \cite{Mantica:2019} shows Einstein-like generalized
Robertson-Walker space-times are perfect fluid space-times except those in
class $\mathcal{I}$ which are not restricted. Sufficient conditions on
generalized Robertson-Walker space-times in this class to be a perfect fluid
are derived in \cite{De:2019}.

Doubly warped products is a generalization of singly warped products
introduced in \cite{Bishop:1969}. The geometric properties of doubly warped
product manifolds have been investigated by many authors such as
pseudo-convexity in \cite{Allison:1991}, harmonic Weyl conformal curvature
tensor in \cite{Gebarowski:1993}, conformal flatness in \cite%
{Gebarowski:1995,Gebarowski:1996}, geodesic completeness in \cite{Unal:2001}%
, doubly warped product submanifolds in \cite%
{Faghfouri:2015,Olteanu:2010,Olteanu:2014,Perktas:2010} and conformal vector
fields in \cite{Elsayied:2016}. Doubly warped space-times are widely used as
exact solutions of Einstein's field equations. Recently, the existence of
compact Einstein doubly warped product manifolds is considered in \cite%
{Gupta:2018}.

Inspired by the above studies of Einstein-like metrics and doubly warped
product manifolds, we studied doubly warped product manifolds equipped with
Einstein-like metrics. The inheritance properties of the Einstein-like class
type $\mathcal{P}$, $\mathcal{A},$ $\mathcal{B}$, $\mathcal{I\oplus A}$, $%
\mathcal{I\oplus B}$\ or $\mathcal{A\oplus B}$ are investigated. To assure
that factor manifolds of a doubly warped product manifold inherits the
Einstein-like class type, sufficient and necessary conditions are derived on
the warping functions. Finally, we apply the results to doubly warped
space-times.

\section{Preliminaries}

A doubly warped product manifold is the (pseudo-)Riemannian product manifold 
$M=M_{1}\times M_{2}$ of two (pseudo-)Riemannian manifolds $%
(M_{i},g_{i},D_{i}),i=1,2,$ furnished with the metric tensor%
\begin{equation*}
g=\left( f_{2}\circ \pi _{2}\right) ^{2}\pi _{1}^{\ast }\left( g_{1}\right)
\oplus \left( f_{1}\circ \pi _{1}\right) ^{2}\pi _{2}^{\ast }\left(
g_{2}\right) ,
\end{equation*}%
where the functions $f_{i}:M_{i}\rightarrow \left( 0,\infty \right) ,$ $%
i=1,2 $ are the warping functions of $M$. $M$ is denoted by $%
_{f_{2}}M_{1}\times _{f_{1}}M_{2}$. The maps $\pi _{i}:M_{1}\times
M_{2}\rightarrow M_{i}$ are the natural projections $M$ onto $M_{i}$ whereas 
$^{\ast }$ denotes the pull-back operator on tensors. In particular, if for
example $f_{2}=1$, then $M=M_{1}\times _{f_{1}}M_{2}$ is called a singly
warped product manifold (see \cite{Elsayied:2016,Unal:2001} for doubly
warped products and \cite%
{Bishop:1969,De:2019B,Elsayied:2017,Oneill:1983,Shenawy:2015,Shenawy:2016A}
for singly warped products).

\begin{notation}
Throughout this work, we use the following notations

\begin{enumerate}
\item All tensor fields on $M_{i}$ are identified with their lifts to $M$.
For example, we use $f_{i}$ for a function on $M_{i}$ and for its lift $%
\left( f_{i}\circ \pi _{i}\right) $ on $M$.

\item The manifolds $M_{i}$ has dimensions $n_{i}$ where $n=n_{1}+n_{2}$.

\item {$\mathrm{Ric}$} is the Ricci curvature tensor on $M$ and {$\mathrm{Ric%
}$}$^{i}$ is the Ricci tensor on $M^{i}$.

\item The gradient of $f_{i}$ on $M_{i}$ is denoted by $\nabla ^{i}f_{i}$
and the Laplacian by $\bigtriangleup ^{i}f_{i}$ whereas $f_{i}^{\diamond
}=f_{i}\bigtriangleup ^{i}f_{i}+\left( n_{j}-1\right) g_{i}\left( \nabla
^{i}f_{i},\nabla ^{i}f_{i}\right) ,$ $i\neq j$.

\item The indices $i$ and $j$ to denote the geometric objects of the factor
manifolds $M_{i}$ and $M_{j}$.

\item The $\left( 0,2\right) $ tensors $\mathcal{F}^{i}$ is defined as%
\begin{equation*}
\mathcal{F}^{i}\left( X_{i},Y_{i}\right) =\frac{n_{j}}{f_{i}}H^{f_{i}}\left(
X_{i},Y_{i}\right) ,
\end{equation*}%
for $X_{i},Y_{i}\in \mathfrak{X}\left( M_{i}\right) $ and $i,j=1,2,i\neq j$.
\end{enumerate}
\end{notation}

The Levi-Civita connection $D$ on $M=_{f_{2}}M_{1}\times _{f_{1}}M_{2}$ is
given by%
\begin{eqnarray*}
D_{X_{i}}X_{j} &=&X_{i}\left( \ln f_{i}\right) X_{j}+X_{j}\left( \ln
f_{j}\right) X_{i}, \\
D_{X_{i}}Y_{i} &=&D_{X_{i}}^{i}Y_{i}-\dfrac{f_{j}^{2}}{f_{i}^{2}}g_{i}\left(
X_{i},Y_{i}\right) \nabla ^{j}(\ln f_{j}),
\end{eqnarray*}%
where $i\neq j$ and $X_{i},Y_{i}\in \mathfrak{X}\left( M_{i}\right) $. Then
the Ricci curvature tensor {$\mathrm{Ric}$} on $M$ is given by%
\begin{eqnarray*}
{\mathrm{Ric}}\left( X_{i},Y_{i}\right) &=&{\mathrm{Ric}}^{i}\left(
X_{i},Y_{i}\right) -\frac{n_{j}}{f_{i}}H^{f_{i}}\left( X_{i},Y_{i}\right) -%
\frac{f_{j}^{\diamond }}{f_{i}^{2}}g_{i}\left( X_{i},Y_{i}\right) , \\
{\mathrm{Ric}}\left( X_{i},Y_{j}\right) &=&\left( n-2\right) X_{i}\left( \ln
f_{i}\right) Y_{j}\left( \ln f_{j}\right) ,
\end{eqnarray*}%
where $i\neq j$ and $X_{i},Y_{i}\in \mathfrak{X}\left( M_{i}\right) $. The
reader is referred to \cite{De:2010,Chojnacka:2013,Gebarowski:1994} for some
studies of curvature conditions on warped product manifolds.

\section{Einstein-like doubly warped product manifolds}

The Einstein-like doubly warped product manifolds $M=_{f_{2}}M_{1}\times
_{f_{1}}M_{2}$ are investigated in this section. Every subsection is devoted
to the study of a class of Einstein-like doubly warped product manifolds.
Sufficient and necessary conditions are derived on the warping functions $%
f_{i}$ for factor manifolds $M_{i}$ to acquire the same Einstein-like class
type.

\subsection{Class $\mathcal{A}$}

A doubly warped product manifold $\left( M,g\right) $ whose Ricci tensor is
Killing, that is,%
\begin{equation*}
\left( D_{X}\mathrm{Ric}\right) \left( Y,Z\right) +\left( D_{Y}\mathrm{Ric}%
\right) \left( Z,X\right) +\left( D_{Z}\mathrm{Ric}\right) \left( X,Y\right)
=0,
\end{equation*}%
for any vector fields $X,Y,Z\in \mathfrak{X}\left( M\right) $ is called
Einstein-like doubly warped product manifold of class $\mathcal{A}$. This
condition equivalent to%
\begin{equation*}
\left( D_{X}\mathrm{Ric}\right) \left( X,X\right) =0,
\end{equation*}%
for any vector field $X\in \mathfrak{X}\left( M\right) $ and the Ricci
tensor is also called cyclic parallel. The legacy of factor manifolds of $M$
in class $\mathcal{A}$ is as follows.

\begin{theorem}
\label{TH1}In a doubly warped product manifold $M=_{f_{2}}M_{1}\times
_{f_{1}}M_{2}$ where $M$ is of class type $\mathcal{A}$, a factor manifold $%
\left( M_{i},g_{i}\right) $ is an Einstein-like manifold of class $\mathcal{A%
}$ if and only if%
\begin{equation*}
\left( D_{X_{i}}^{i}\mathcal{F}^{i}\right) \left( X_{i},X_{i}\right) =\frac{2%
}{f_{i}^{3}}X_{i}\left( f_{i}\right) g_{i}\left( X_{i},X_{i}\right) \left[
f_{j}^{\diamond }+\left( n-2\right) \left( \nabla ^{j}f_{j}\right) \left(
f_{j}\right) \right] ,
\end{equation*}%
where $i,j=1,2,i\neq j$ and $X_{i}\in \mathfrak{X}\left( M_{i}\right) $.
\end{theorem}

\begin{proof}
In a doubly warped product manifold $M=_{f_{2}}M_{1}\times _{f_{1}}M_{2}$ of
class $\mathcal{A}$, it is%
\begin{eqnarray*}
0 &=&\left( D_{X}\mathrm{Ric}\right) \left( X,X\right) \\
&=&X\left( \mathrm{Ric}\left( X,X\right) \right) -2\mathrm{Ric}\left(
D_{X}X,X\right) .
\end{eqnarray*}%
Thus, for a the special case where $X=X_{i}$ lands on one factor, one may get%
\begin{eqnarray*}
0 &=&\left( D_{X_{i}}\mathrm{Ric}\right) \left( X_{i},X_{i}\right) \\
&=&X_{i}\left( {\mathrm{Ric}}^{i}\left( X_{i},X_{i}\right) -\mathcal{F}%
^{i}\left( X_{i},X_{i}\right) -\frac{f_{j}^{\diamond }}{f_{i}^{2}}%
g_{i}\left( X_{i},X_{i}\right) \right) \\
&&-2\mathrm{Ric}^{i}\left( D_{X_{i}}^{i}X_{i},X_{i}\right) +2\mathcal{F}%
^{i}\left( D_{X_{i}}^{i}X_{i},X_{i}\right) +2\frac{f_{j}^{\diamond }}{%
f_{i}^{2}}g_{i}\left( D_{X_{i}}^{i}X_{i},X_{i}\right) \\
&&+2\left( n-2\right) \dfrac{1}{f_{i}^{3}}\left( \nabla ^{j}f_{j}\right)
\left( f_{j}\right) X_{i}\left( f_{i}\right) g_{i}\left( X_{i},X_{i}\right) .
\end{eqnarray*}%
Thus, after lengthy computations, it is%
\begin{eqnarray*}
0 &=&\left( D_{X_{i}}^{i}{\mathrm{Ric}}^{i}\right) \left( X_{i},X_{i}\right)
-\left( D_{X_{i}}^{i}\mathcal{F}^{i}\right) \left( X_{i},X_{i}\right) \\
&&+\frac{2}{f_{i}^{3}}X_{i}\left( f_{i}\right) g_{i}\left(
X_{i},X_{i}\right) \left[ f_{j}^{\diamond }+\left( n-2\right) \left( \nabla
^{j}f_{j}\right) \left( f_{j}\right) \right] .
\end{eqnarray*}%
These equations complete the proof.
\end{proof}

It is now easy to recover a similar result on singly warped product
manifolds.

\begin{corollary}
In a singly warped product manifold $M=M_{1}\times _{f_{1}}M_{2}$ where $M$
is of class type $\mathcal{A}$, $\left( M_{1},g_{1}\right) $ is an
Einstein-like manifold of class $\mathcal{A}$ if and only if $\mathcal{F}%
^{i} $ is Killing. In addition, $\left( M_{2},g_{2}\right) $ is of class
type $\mathcal{A}$.
\end{corollary}

\subsection{Class $\mathcal{B}$}

Let $M$ be as Einstein-like doubly warped product manifold of class $%
\mathcal{B}$. Then, the Ricci tensor is a Codazzi tensor,that is,%
\begin{equation*}
\left( D_{X}\mathrm{Ric}\right) \left( Y,Z\right) =\left( D_{Y}\mathrm{Ric}%
\right) \left( X,Z\right) .
\end{equation*}%
The above condition is equivalent to:

\begin{enumerate}
\item $M$ has a harmonic Riemann tensor, that is, $\nabla _{\varepsilon }%
\mathcal{R}_{\alpha \beta \gamma }^{\text{ \ \ }\varepsilon }=0$, or

\item $M$ admits a harmonic Weyl conformal tensor and the scalar curvature
is constant, that is, $\nabla _{\varepsilon }\mathcal{C}_{\alpha \beta
\gamma }^{\text{ \ \ }\varepsilon }=0$ and $\nabla _{\varepsilon }R=0$.
\end{enumerate}

The base manifold and the fiber manifold gain the Einstein-like class type $%
\mathcal{B}$ according to.

\begin{theorem}
\label{TH2}In a doubly warped product manifold $M=_{f_{2}}M_{1}\times
_{f_{1}}M_{2}$ where $M$ is of class type $\mathcal{B}$, the factor manifold 
$\left( M_{i},g_{i}\right) $ is an Einstein-like manifold of class $\mathcal{%
B}$ if and only if%
\begin{eqnarray*}
\left( D_{X_{i}}^{i}\mathcal{F}^{i}\right) \left( Y_{i},Z_{i}\right)
&=&\left( D_{Y_{i}}^{i}\mathcal{F}^{i}\right) \left( X_{i},Z_{i}\right) \\
&&+\frac{1}{f_{i}^{3}}X_{i}\left( f_{i}\right) g_{i}\left(
Y_{i},Z_{i}\right) \left( 2f_{j}^{\diamond }-\left( n-2\right) \left( \nabla
^{j}f_{j}\right) f_{j}\right) \\
&&-\frac{1}{f_{i}^{3}}Y_{i}\left( f_{i}\right) g_{i}\left(
X_{i},Z_{i}\right) \left( 2f_{j}^{\diamond }-\left( n-2\right) \left( \nabla
^{j}f_{j}\right) f_{j}\right) ,
\end{eqnarray*}%
where $i,j=1,2,i\neq j$ and $X_{i},Y_{i},Z_{i}\in \mathfrak{X}\left(
M_{i}\right) $.
\end{theorem}

\begin{proof}
Let us define the deviation tensor $B\left( X,Y,Z\right) $ as follows%
\begin{equation*}
B\left( X,Y\right) Z=\left( D_{X}\mathrm{Ric}\right) \left( Y,Z\right)
-\left( D_{Y}\mathrm{Ric}\right) \left( X,Z\right) .
\end{equation*}%
There are three different cases. Let us consider the first case, that is,%
\begin{equation}
B\left( X_{i},Y_{i},Z_{i}\right) =\left( D_{X_{i}}\mathrm{Ric}\right) \left(
Y_{i},Z_{i}\right) -\left( D_{Y_{i}}\mathrm{Ric}\right) \left(
X_{i},Z_{i}\right) .  \label{L1}
\end{equation}%
It is enough to find $\left( D_{X_{i}}\mathrm{Ric}\right) \left(
Y_{i},Z_{i}\right) $ as%
\begin{eqnarray*}
\left( D_{X_{i}}\mathrm{Ric}\right) \left( Y_{i},Z_{i}\right) &=&X_{i}\left( 
{\mathrm{Ric}}^{i}\left( Y_{i},Z_{i}\right) \right) -X_{i}\left( \mathcal{F}%
^{i}\left( Y_{i},Z_{i}\right) \right) -f_{j}^{\diamond }X_{i}\left( \frac{1}{%
f_{i}^{2}}g_{i}\left( Y_{i},Z_{i}\right) \right) \\
&&-{\mathrm{Ric}}^{i}\left( D_{X_{i}}^{i}Y_{i},Z_{i}\right) +\mathcal{F}%
^{i}\left( D_{X_{i}}^{i}Y_{i},Z_{i}\right) +\frac{f_{j}^{\diamond }}{%
f_{i}^{2}}g_{i}\left( D_{X_{i}}^{i}Y_{i},Z_{i}\right) \\
&&-{\mathrm{Ric}}^{i}\left( Y_{i},D_{X_{i}}^{i}Z_{i}\right) +\mathcal{F}%
^{i}\left( Y_{i},D_{X_{i}}^{i}Z_{i}\right) +\frac{f_{j}^{\diamond }}{%
f_{i}^{2}}g_{i}\left( Y_{i},D_{X_{i}}^{i}Z_{i}\right) \\
&&+\left( n-2\right) \dfrac{1}{f_{i}^{3}}g_{i}\left( X_{i},Y_{i}\right)
\nabla ^{j}f_{j}\left( f_{j}\right) Z_{i}\left( f_{i}\right) \\
&&+\left( n-2\right) \dfrac{1}{f_{i}^{3}}g_{i}\left( X_{i},Z_{i}\right)
\nabla ^{j}f_{j}\left( f_{j}\right) Y_{i}\left( f_{i}\right) .
\end{eqnarray*}%
Simplifying this expression, it is%
\begin{eqnarray}
\left( D_{X_{i}}\mathrm{Ric}\right) \left( Y_{i},Z_{i}\right) &=&\left(
D_{X_{i}}^{i}\mathrm{Ric}^{i}\right) \left( Y_{i},Z_{i}\right) -\left(
D_{X_{i}}^{i}\mathcal{F}^{i}\right) \left( Y_{i},Z_{i}\right) +2\frac{%
f_{j}^{\diamond }}{f_{i}^{3}}X_{i}\left( f_{i}\right) g_{i}\left(
Y_{i},Z_{i}\right)  \notag \\
&&+\left( n-2\right) \dfrac{1}{f_{i}^{3}}g_{i}\left( X_{i},Y_{i}\right)
\nabla ^{j}f_{j}\left( f_{j}\right) Z_{i}\left( f_{i}\right)  \notag \\
&&+\left( n-2\right) \dfrac{1}{f_{i}^{3}}g_{i}\left( X_{i},Z_{i}\right)
\nabla ^{j}f_{j}\left( f_{j}\right) Y_{i}\left( f_{i}\right) .  \label{L2}
\end{eqnarray}%
By exchanging $X_{i}$ and $Y_{i}$ in the last equation and substitution in
Equation (\ref{L1}), one gets the deviation tensor. For Einstein-like
manifolds of class $\mathcal{B}$, the deviation tensor vanishes from which
the result hold.
\end{proof}

It is easy to retrieve a similar result on a singly warped product manifold.

\begin{corollary}
In a singly warped product manifold $M=M_{1}\times _{f_{1}}M_{2}$ where $M$
is of class type $\mathcal{B}$, $\left( M_{1},g_{1}\right) $ is an
Einstein-like manifold of class $\mathcal{B}$ if and only if%
\begin{equation*}
\left( D_{X_{1}}^{1}\mathcal{F}^{1}\right) \left( Y_{1},Z_{1}\right) =\left(
D_{Y_{1}}^{1}\mathcal{F}^{1}\right) \left( X_{1},Z_{1}\right) ,
\end{equation*}%
where $X_{1},Y_{1},Z_{1}\in \mathfrak{X}\left( M_{1}\right) $. In addition, $%
\left( M_{2},g_{2}\right) $ is Einstein-like of class type $\mathcal{B}$.
\end{corollary}

\subsection{Class $\mathcal{P}$}

Let $M$ be an Einstein-like doubly warped product manifold of class $%
\mathcal{P}$. Thus, $M$ has a parallel Ricci tensor, that is,%
\begin{equation*}
\left( D_{X}\mathrm{Ric}\right) \left( Y,Z\right) =0.
\end{equation*}%
Manifolds in this class are usually called Ricci symmetric.

\begin{theorem}
\label{TH3}In a doubly warped product manifold $M=_{f_{2}}M_{1}\times
_{f_{1}}M_{2}$ where $M$ is of class type $\mathcal{P}$, $\left(
M_{i},g_{i}\right) $ is an Einstein-like manifold of class $\mathcal{P}$ if
and only if%
\begin{eqnarray*}
\left( D_{X_{i}}^{i}\mathcal{F}^{i}\right) \left( Y_{i},Z_{i}\right) &=&%
\dfrac{n-2}{f_{i}^{3}}\left[ g_{i}\left( X_{i},Y_{i}\right) Z_{i}\left(
f_{i}\right) +g_{i}\left( X_{i},Z_{i}\right) Y_{i}\left( f_{i}\right) \right]
\left( \nabla ^{j}f_{j}\right) f_{j} \\
&&+2\frac{f_{j}^{\diamond }}{f_{i}^{3}}X_{i}\left( f_{i}\right) g_{i}\left(
Y_{i},Z_{i}\right) ,
\end{eqnarray*}%
where $i,j=1,2,i\neq j$ and $X_{i},Y_{i},Z_{i}\in \mathfrak{X}\left(
M_{i}\right) $.
\end{theorem}

\begin{proof}
Let $M=_{f_{2}}M_{1}\times _{f_{1}}M_{2}$ be a Ricci symmetric doubly warped
product manifold, that is,%
\begin{equation*}
0=\left( D_{X}\mathrm{Ric}\right) \left( Y,Z\right)
\end{equation*}%
Equation (\ref{L2}) infers%
\begin{eqnarray*}
\left( D_{X_{i}}\mathrm{Ric}\right) \left( Y_{i},Z_{i}\right) &=&\left(
D_{X_{i}}^{i}\mathrm{Ric}^{i}\right) \left( Y_{i},Z_{i}\right) -\left(
D_{X_{i}}^{i}\mathcal{F}^{i}\right) \left( Y_{i},Z_{i}\right) +2\frac{%
f_{j}^{\diamond }}{f_{i}^{3}}X_{i}\left( f_{i}\right) g_{i}\left(
Y_{i},Z_{i}\right) \\
&&+\left( n-2\right) \dfrac{1}{f_{i}^{3}}g_{i}\left( X_{i},Y_{i}\right)
\nabla ^{j}f_{j}\left( f_{j}\right) Z_{i}\left( f_{i}\right) \\
&&+\left( n-2\right) \dfrac{1}{f_{i}^{3}}g_{i}\left( X_{i},Z_{i}\right)
\nabla ^{j}f_{j}\left( f_{j}\right) Y_{i}\left( f_{i}\right) .
\end{eqnarray*}%
Thus, having a parallel Ricci tensor implies%
\begin{eqnarray*}
\left( D_{X_{i}}^{i}\mathrm{Ric}^{i}\right) \left( Y_{i},Z_{i}\right)
&=&\left( D_{X_{i}}^{i}\mathcal{F}^{i}\right) \left( Y_{i},Z_{i}\right) -2%
\frac{f_{j}^{\diamond }}{f_{i}^{3}}X_{i}\left( f_{i}\right) g_{i}\left(
Y_{i},Z_{i}\right) \\
&&-\dfrac{n-2}{f_{i}^{3}}\left[ g_{i}\left( X_{i},Y_{i}\right) Z_{i}\left(
f_{i}\right) +g_{i}\left( X_{i},Z_{i}\right) Y_{i}\left( f_{i}\right) \right]
\nabla ^{j}f_{j}\left( f_{j}\right) .
\end{eqnarray*}%
This equation completes the proof.
\end{proof}

The corresponding result on singly warped product manifolds is as follows.

\begin{corollary}
In a singly warped product manifold $M=M_{1}\times _{f_{1}}M_{2}$ where $M$
is of class type $\mathcal{P}$. Then $\left( M_{1},g_{1}\right) $ is an
Einstein-like manifold of class $\mathcal{P}$ if and only if%
\begin{equation*}
\left( D_{X_{1}}^{1}\mathcal{F}^{1}\right) \left( Y_{1},Z_{1}\right) =0,
\end{equation*}%
where $X_{1},Y_{1},Z_{1}\in \mathfrak{X}\left( M_{1}\right) $. Also, $\left(
M_{2},g_{2}\right) $ is Einstein-like of class type $\mathcal{P}$.
\end{corollary}

\subsection{Class $\mathcal{I\oplus B}$}

A doubly warped product manifold $M$ is of class type $\mathcal{I\oplus B}$
if its Ricci tensor satisfies%
\begin{equation*}
\nabla _{\gamma }\left[ R_{\alpha \beta }-\frac{R}{2\left( n-1\right) }%
g_{\alpha \beta }\right] =\nabla _{\alpha }\left[ R_{\gamma \beta }-\frac{R}{%
2\left( n-1\right) }g_{\gamma \beta }\right] ,
\end{equation*}%
that is, the tensor $\mathcal{H}_{\alpha \beta }=R_{\alpha \beta }-\frac{R}{%
2\left( n-1\right) }g_{\alpha \beta }$ is a Codazzi tensor. This condition
is equivalent to%
\begin{equation*}
\nabla _{\varepsilon }\mathcal{C}_{\alpha \beta \gamma }^{\text{ \ \ }%
\varepsilon }=0,
\end{equation*}%
where $\mathcal{C}$ is the Weyl conformal curvature tensor and $n\geq 3$,
i.e., $M$ has a harmonic Weyl tensor. Let $g_{\beta \gamma }=\varphi ^{2}%
\bar{g}_{\beta \gamma }$ be a conformal change of on a manifold $M$. It is
well known that the Weyl tensor $\mathcal{C}_{\alpha \beta \gamma }^{\text{
\ \ }\varepsilon }$ remains invariant, that is, $\mathcal{\bar{C}}_{\alpha
\beta \gamma }^{\text{ \ \ }\varepsilon }=\mathcal{C}_{\alpha \beta \gamma
}^{\text{ \ \ }\varepsilon }$ however $\mathcal{C}_{\alpha \beta \gamma
\varepsilon }=\varphi ^{2}\mathcal{\bar{C}}_{\alpha \beta \gamma \varepsilon
}$. The divergence of the Weyl tensor is given by\cite{Besse:2008}%
\begin{equation}
\nabla _{\varepsilon }\mathcal{C}_{\alpha \beta \gamma }^{\text{ \ \ }%
\varepsilon }=\bar{\nabla}_{\varepsilon }\mathcal{\bar{C}}_{\alpha \beta
\gamma }^{\text{ \ \ }\varepsilon }-\frac{n-3}{\varphi }\left( \nabla
_{\varepsilon }\varphi \right) \mathcal{\bar{C}}_{\alpha \beta \gamma }^{%
\text{ \ \ }\varepsilon }.  \label{C02}
\end{equation}%
The doubly warped product metric may be rewritten as follows%
\begin{eqnarray*}
g &=&f_{1}^{2}f_{2}^{2}\left( f_{1}^{-2}g_{1}+f_{2}^{-2}g_{2}\right) \\
&=&f_{1}^{2}f_{2}^{2}\left( \bar{g}_{1}+\bar{g}_{2}\right) \\
&=&f_{1}^{2}f_{2}^{2}\bar{g}
\end{eqnarray*}%
where $g_{i}=f_{i}^{2}\bar{g}_{i}$ and $\bar{g}=\bar{g}_{1}+\bar{g}_{2}$.
The doubly warped product manifold $\left( M,g\right) $ has harmonic Weyl
tensor if and only%
\begin{equation}
\bar{\nabla}_{\varepsilon }\mathcal{\bar{C}}_{\alpha \beta \gamma }^{\text{
\ \ }\varepsilon }=\frac{n-3}{\varphi }\left( \nabla _{\varepsilon }\varphi
\right) \mathcal{C}_{\alpha \beta \gamma }^{\text{ \ \ }\varepsilon }
\label{C03}
\end{equation}%
where $\varphi =f_{1}f_{2}$. Assume that $\nabla _{\varepsilon }\left(
f_{1}f_{2}\right) \mathcal{C}_{\alpha \beta \gamma }^{\text{ \ \ }%
\varepsilon }=0$, then%
\begin{equation}
\bar{\nabla}_{\varepsilon }\mathcal{\bar{C}}_{\alpha \beta \gamma }^{\text{
\ \ }\varepsilon }=0.  \label{C01}
\end{equation}%
having a harmonic Weyl tensor is equivalent to the condition%
\begin{eqnarray*}
0 &=&\mathcal{\bar{T}}_{\alpha \beta \gamma } \\
&=&\bar{\nabla}_{\gamma }\bar{R}_{\alpha \beta }-\bar{\nabla}_{\gamma }\bar{R%
}_{\alpha \beta }-\frac{1}{2\left( n-1\right) }\left[ \left( \bar{\nabla}%
_{\gamma }\bar{R}\right) \bar{g}_{\alpha \beta }-\left( \bar{\nabla}_{\gamma
}\bar{R}\right) \bar{g}_{\alpha \beta }\right] ,
\end{eqnarray*}%
where $\mathcal{\bar{T}}$ is the Cotton tensor. The metric $\bar{g}$ splits
as $\bar{g}=\bar{g}_{1}+\bar{g}_{2}$ and consequently the divergence of the
Cotton tensor $\mathcal{\bar{T}}$ splits on the factor manifolds $\left(
M_{i},\bar{g}_{i}\right) $ as%
\begin{equation}
0=\mathcal{\bar{T}}_{\alpha \beta \gamma }^{i}+\frac{n_{2}}{2\left(
n-1\right) \left( n_{1}-1\right) }\left[ \left( \bar{\nabla}_{\gamma }^{i}%
\bar{R}^{i}\right) \left( \bar{g}_{i}\right) _{\alpha \beta }-\left( \bar{%
\nabla}_{\gamma }^{i}\bar{R}^{i}\right) \left( \bar{g}_{i}\right) _{\alpha
\beta }\right] .  \label{C04}
\end{equation}%
In this case, $\left( \bar{\nabla}_{\gamma }^{i}\bar{R}^{i}\right) \left( 
\bar{g}_{i}\right) _{\alpha \beta }-\left( \bar{\nabla}_{\gamma }^{i}\bar{R}%
^{i}\right) \left( \bar{g}_{i}\right) _{\alpha \beta }=0$, that is, $\bar{R}%
^{i}$ is constant if and only if the cotton tensor $\mathcal{\bar{T}}^{i}$
on the doubly warped factor manifolds $\left( M^{i},\bar{g}_{i}\right) $
vanishes i.e. 
\begin{equation*}
\bar{\nabla}_{\varepsilon }\mathcal{\bar{C}}_{\alpha \beta \gamma }^{i\text{
\ \ }\varepsilon }=0.
\end{equation*}%
The Weyl tensors $\mathcal{C}^{i}$ on doubly warped product factor manifolds 
$\left( M_{i},g_{i}\right) $ satisfy%
\begin{eqnarray}
0 &=&\bar{\nabla}_{\varepsilon }\mathcal{\bar{C}}_{\alpha \beta \gamma }^{i%
\text{ \ \ }\varepsilon }  \notag \\
&=&\nabla _{\varepsilon }\mathcal{C}_{\alpha \beta \gamma }^{i\text{ \ \ }%
\varepsilon }+\frac{n_{i}-3}{f_{i}}\left( \nabla _{\varepsilon
}^{i}f_{i}\right) \mathcal{C}_{\alpha \beta \gamma }^{i\text{ \ \ }%
\varepsilon }.  \label{C05}
\end{eqnarray}%
It is time now to write the following result.

\begin{theorem}
In a doubly warped product manifold $M=_{f_{2}}M_{1}\times _{f_{1}}M_{2}$
where $M$ is of class type $\mathcal{I\oplus B}$. Assume that $\nabla
_{\varepsilon }\left( f_{1}f_{2}\right) \mathcal{C}_{\alpha \beta \gamma }^{%
\text{ \ \ }\varepsilon }=0$ and the conformal change $\left(
M_{i},f_{i}^{-2}g_{i}\right) $ has a constant scalar curvature. Then $\left(
M_{i},g_{i}\right) $ is an Einstein-like manifold of class $\mathcal{I\oplus
B}$ if and only if $\left( \nabla _{\varepsilon }^{i}f_{i}\right) \mathcal{C}%
_{\alpha \beta \gamma }^{i\text{ \ \ }\varepsilon }=0$ for each $i=1,2$.
\end{theorem}

A. Gebarowski proved an inheritance property of this class in \cite[Theorem 2%
]{Gebarowski:1993}.

\subsection{Class $\mathcal{I\oplus A}$}

Doubly warped product manifolds where the tensor%
\begin{equation*}
\mathcal{L}=\mathrm{Ric}-\frac{2R}{n+2}g
\end{equation*}%
is Killing lies the class $\mathcal{I\oplus A}$. The above condition is
equivalent to%
\begin{equation*}
0=\left( D_{X}\mathcal{L}\right) \left( X,X\right) .
\end{equation*}%
The following theorem draw the inheritance property of this class.

\begin{theorem}
In a doubly warped product manifold $M=_{f_{2}}M_{1}\times _{f_{1}}M_{2}$
where $M$ is of class type $\mathcal{I\oplus A}$, the factor manifold $%
\left( M_{i},g_{i}\right) $ is of class type $\mathcal{I\oplus A}$ if and
only if%
\begin{eqnarray*}
\left( D_{X_{i}}^{i}\mathcal{F}^{i}\right) \left( X_{i},X_{i}\right) &=&%
\frac{2}{f_{i}^{3}}X_{i}\left( f_{i}\right) g_{i}\left( X_{i},X_{i}\right) %
\left[ f_{j}^{\diamond }+\left( n-2\right) \left( \nabla ^{j}f_{j}\right)
\left( f_{j}\right) \right] \\
&&-\frac{2}{n+2}\left( D_{X_{i}}R-\frac{n+2}{n_{i}+2}D_{X_{i}}^{i}R^{i}%
\right) g_{i}\left( X_{i},X_{i}\right) .
\end{eqnarray*}
\end{theorem}

\begin{proof}
Assume that $M=_{f_{2}}M_{1}\times _{f_{1}}M_{2}$ be a doubly warped product
manifold of class type $\mathcal{I\oplus A}$. Then%
\begin{eqnarray*}
0 &=&\left( D_{X}\right) \left( \mathrm{Ric}-\frac{2R}{n+2}g\right) \left(
X,X\right) \\
&=&\left( D_{X}\mathrm{Ric}\right) \left( X,X\right) -\frac{2}{n+2}g\left(
X,X\right) D_{X}R.
\end{eqnarray*}%
Using equation (\ref{L2}), it is%
\begin{eqnarray*}
0 &=&\left( D_{X_{i}}^{i}{\mathrm{Ric}}^{i}\right) \left( X_{i},X_{i}\right)
-\left( D_{X_{i}}^{i}\mathcal{F}^{i}\right) \left( X_{i},X_{i}\right) \\
&&+\frac{2}{f_{i}^{3}}X_{i}\left( f_{i}\right) g_{i}\left(
X_{i},X_{i}\right) \left[ f_{j}^{\diamond }+\left( n-2\right) \left( \nabla
^{j}f_{j}\right) \left( f_{j}\right) \right] \\
&&-\frac{2}{n+2}\left( D_{Xi}R\right) g_{i}\left( X_{i},X_{i}\right)
\end{eqnarray*}%
and consequently, one has%
\begin{eqnarray*}
0 &=&\left( D_{X_{i}}^{i}{\mathrm{Ric}}^{i}\right) \left( X_{i},X_{i}\right)
-\frac{2}{n_{i}+2}g_{i}\left( X_{i},X_{i}\right) D_{X_{i}}^{i}R^{i} \\
&&-\left( D_{X_{i}}^{i}\mathcal{F}^{i}\right) \left( X_{i},X_{i}\right) \\
&&+\frac{2}{f_{i}^{3}}X_{i}\left( f_{i}\right) g_{i}\left(
X_{i},X_{i}\right) \left[ f_{j}^{\diamond }+\left( n-2\right) \left( \nabla
^{j}f_{j}\right) \left( f_{j}\right) \right] \\
&&-\frac{2}{n+2}\left( D_{Xi}R-\frac{n+2}{n_{i}+2}D_{X_{i}}^{i}R^{i}\right)
g_{i}\left( X_{i},X_{i}\right)
\end{eqnarray*}%
which completes the proof.
\end{proof}

\subsection{Class $\mathcal{A\oplus B}$}

This class is identified by having a constant scalar curvature. Let $%
M=_{f_{2}}M_{1}\times _{f_{1}}M_{2}$ be a doubly warped product manifold of
class type $\mathcal{A\oplus B}$, that is, the scalar curvature $R$ of $M$
is constant, say $c$. The use of Equation 7 in \cite{Gebarowski:1993}
implies that $M_{i}$ is of class $\mathcal{A\oplus B}$ if there are two
constants $c_{i}$ and $c_{j}$ such that%
\begin{equation*}
\frac{c_{i}}{f_{j}^{2}}+\frac{c_{j}}{f_{i}^{2}}-\frac{n_{i}\left(
n_{i}-1\right) }{f_{j}^{2}}\Delta _{j}f_{j}-\frac{n_{j}\left( n_{j}-1\right) 
}{f_{i}^{2}}\Delta _{i}f_{i}-\frac{2n_{i}}{f_{j}}F_{j}-\frac{2n_{j}}{f_{i}}%
F_{i}=c,
\end{equation*}%
where $F_{i}=g_{i}^{\alpha \beta }\nabla _{\alpha }^{i}\nabla _{\beta
}^{i}f_{i}$.

\section{Einstein-like doubly warped Relativistic space-times}

Let $(M,g)$ be a Riemannian manifold, $f:M\rightarrow \left( 0,\infty
\right) $ and $\sigma :I\rightarrow \left( 0,\infty \right) $ are smooth
functions. The manifold $\bar{M}=_{f}I\times _{\sigma }M$ furnished with the
metric tensor $\bar{g}=-f^{2}dt^{2}\oplus \sigma ^{2}g$ is called a doubly
warped space-time. For $U,V\in \mathfrak{X}\left( M\right) ,$ the covariant
derivative $\bar{D}$ on $\bar{M}$ is given by%
\begin{eqnarray*}
\bar{D}_{\partial _{t}}\partial _{t} &=&\dfrac{f}{\sigma ^{2}}\nabla f, \\
\bar{D}_{\partial _{t}}U &=&D_{U}\partial _{t}=\frac{\dot{\sigma}}{\sigma }U+%
\frac{1}{f}U\left( f\right) \partial _{t}, \\
\bar{D}_{U}V &=&D_{U}V-\dfrac{\sigma \dot{\sigma}}{f^{2}}g\left( U,V\right)
\partial _{t},
\end{eqnarray*}%
whereas the Ricci tensor {$\mathrm{\bar{R}ic}$} on $\bar{M}$ is given by%
\begin{eqnarray*}
\mathrm{\bar{R}ic}\left( \partial _{t},\partial _{t}\right) &=&\frac{n}{%
\sigma }\ddot{\sigma}+\frac{f^{\diamond }}{\sigma ^{2}}, \\
\mathrm{\bar{R}ic}\left( U,V\right) &=&\mathrm{Ric}\left( U,V\right) -\frac{1%
}{f}H^{f}\left( U,V\right) -\frac{\sigma ^{\diamond }}{f^{2}}g\left(
U,V\right) , \\
\mathrm{\bar{R}ic}\left( \partial _{t},U\right) &=&\left( n-1\right) \frac{%
\dot{\sigma}}{\sigma }U\left( \ln f\right) .
\end{eqnarray*}%
For the definition and relativistic significance of doubly warped
space-times, the reader is referred to \cite{Elsayied:2016,Ramos:2003} and
references therein.

\begin{theorem}
In a doubly warped space-time $\bar{M}=_{f}I\times _{\sigma }M$ of class
type $\mathcal{A}$, $M$ is an Einstein-like manifold of class type $\mathcal{%
A}$ if and only if%
\begin{equation*}
\left( D_{V}\mathcal{F}\right) \left( V,V\right) =\left( \left( n-1\right) 
\dot{\sigma}^{2}+\sigma ^{\diamond }\right) \frac{2}{f^{3}}V\left( f\right)
g\left( V,V\right) .
\end{equation*}
\end{theorem}

\begin{theorem}
In a doubly warped space-time $\bar{M}=_{f}I\times _{\sigma }M$ of class
type $\mathcal{B}$, $M$ is an Einstein-like manifold of class type $\mathcal{%
B}$ if and only if%
\begin{eqnarray*}
\left( D_{W}\mathcal{F}\right) \left( U,V\right) &=&\left( D_{U}\mathcal{F}%
\right) \left( W,V\right) +\left( 2\sigma ^{\diamond }-\left( n-1\right) 
\dot{\sigma}^{2}\right) \frac{1}{f^{3}}W\left( f\right) g\left( U,V\right) \\
&&-\left( 2\sigma ^{\diamond }+\left( n-1\right) \dot{\sigma}^{2}\right) 
\frac{1}{f^{3}}U\left( f\right) g\left( W,V\right) .
\end{eqnarray*}
\end{theorem}

\begin{theorem}
In a doubly warped space-time $\bar{M}=_{f}I\times _{\sigma }M$ of class
type $\mathcal{P}$, $M$ is an Einstein-like manifold of class type $\mathcal{%
P}$ if and only if%
\begin{equation*}
\left( D_{W}\mathcal{F}\right) \left( U,V\right) =2\frac{\sigma ^{\diamond }%
}{f^{3}}W\left( f\right) g\left( U,V\right) +\dfrac{\dot{\sigma}^{2}}{f^{3}}%
\left( n-1\right) \left( g\left( W,V\right) U\left( f\right) +g\left(
W,U\right) V\left( f\right) \right) .
\end{equation*}
\end{theorem}

\end{document}